\documentclass[11pt]{amsart}
\usepackage[latin1]{inputenc}
\usepackage{color}
\usepackage{listings}
\usepackage{float}
\floatstyle{ruled}
\usepackage{caption}
\usepackage{amsmath}
\usepackage{amssymb}
\usepackage{bbm}
\usepackage{enumerate}
\restylefloat{figure}
\newtheorem{thm}{Theorem}[section]
\newtheorem{cor}[thm]{Corollary}
\newtheorem{lem}[thm]{Lemma}

\theoremstyle{definition}

\numberwithin{equation}{section}

\makeatletter
\@ifpackageloaded{showkeys}{\oddsidemargin=3cm\evensidemargin=3cm}{\oddsidemargin=0.7cm\evensidemargin=0.7cm}\makeatother
\newcommand{\lin}{\operatorname{span}}
\newcommand{\supp}{\operatorname{supp}}

\newcommand{\dif}{\,\mathrm{d}}
\newcommand{\sign}{\operatorname{sign}}
\newcommand{\charfun}{\ensuremath{\mathbbm 1}}

\DeclareMathOperator{\sgn}{sgn}

\allowdisplaybreaks

\begin{document}
\title{Projection operators onto spaces of Chebyshev splines}
\author[K. Keryan]{Karen Keryan}
\address{Yerevan State University, Alex Manoogian 1, 0025 Yerevan, Armenia,
	American University of Armenia, Marshal Baghramyan 40, 0019 Yerevan, Armenia}
\email{karenkeryan@ysu.am, kkeryan@aua.am}

\author[M. Passenbrunner]{Markus Passenbrunner}
\address{Institute of Analysis, Johannes Kepler University Linz, Austria, 4040 Linz, Altenberger Strasse 69}
\email{markus.passenbrunner@jku.at}

\begin{abstract}
 We prove that the Chebyshev spline orthoprojectors are uniformly bounded
on $L^\infty$.
\end{abstract}
\maketitle
\section{Introduction}
In this paper, we extend Shadrin's theorem \cite{Shadrin2001} on the  
boundedness of the polynomial spline orthoprojector on $L^\infty$  by a constant 
that does not depend on the underlying  univariate grid $\Delta$
 to  the setting of Chebyshev
 splines.  The space of Chebyshev splines  $S(\mathcal U_k;\Delta)$
 of order $k$ consists of functions that, on each grid
interval of $\Delta$, are contained in $\mathcal
	U_k=\lin\{u_1,\ldots,u_k\}$, the space of 
linear
combinations of generalized monomials $u_1,\ldots, u_k$ that arise by iterated
integrals of positive weight functions in the same way as the classical
monomials  $1,\ldots,x^{k-1}$ arise as iterated integrals of constant functions.


One of the main reasons to consider this extension to spline
orthoprojectors onto Chebyshev splines is that in recent years, it turned out that in many cases
 (see e.g. \cite{Shadrin2001, PassenbrunnerShadrin2014,
Passenbrunner2014, MuellerPassenbrunner2017, Passenbrunner2017,
KeryanPassenbrunner2017}),
sequences of   orthogonal projections onto classical spline spaces  corresponding to arbitrary 
grid sequences  $\Delta_n$
behave like sequences of conditional
expectations (or, more generally, like  martingales) and we
want to extend martingale type
results to an even larger class of orthogonal projections.

In order to explain those martingale type results, we have to introduce a little bit of
terminology: Let $k$ be a positive integer, $(\mathcal F_n)$ an increasing
sequence of $\sigma$-algebras of sets in $[0,1]$ where each $\mathcal F_n$ is generated 
by a finite partition of $[0,1]$ into intervals of positive length. Moreover,
let
\[
	S_n^{(k)} = \{f\in C^{k-2}[0,1] : f\text{ is a classical polynomial of order $k$ on
	each atom of $\mathcal F_n$} \}
\]
and  define $P_n^{(k)}$ as the orthogonal projection operator onto
$S_n^{(k)}$ with respect to the
$L^2$  inner product on $[0,1]$  with the Lebesgue measure  $|\cdot|$.  The space $S_n^{(1)}$ consists of
piecewise constant functions and $P_n^{(1)}$ is the conditional expectation
operator with respect to the $\sigma$-algebra $\mathcal F_n$.
 Similarly to the definition of martingales, we introduce the following
notion:
let $(f_n)_{n\geq 0}$ be a sequence of integrable functions, we call this
sequence a $k$-martingale spline sequence (adapted to $(\mathcal F_n)$), if
\[
	P_n^{(k)} f_{n+1} = f_n,\qquad n\geq 0.
\]

 Classical martingale theorems such as Doob's
	inequality, the martingale convergence theorem or Burkholder's
	inequality in fact carry over to
	$k$-martingale spline sequences corresponding to  \emph{arbitrary}
	filtrations ($\mathcal F_n$) of the above type. Indeed, we have
		\begin{enumerate}[(i)]
			\item\label{it:splines1} (Shadrin's theorem) 
				there exists
				a constant $C_k$ depending only on $k$ such that 
				\[
					\sup_n\| P_n^{(k)} : L^1 \to L^1 \| \leq
					C_k,
				\]
			\item \label{it:splines2}
				there exists a constant $C_k$  depending only on
				$k$ such that for any $k$-martingale
				spline sequence
				$(f_n)$ and any
				$\lambda>0$, 
				\begin{equation*}
					|\{ \sup_n |f_n| > \lambda \}| \leq C_k
				\frac{\sup_n\|f_n\|_{L^1}}{
				\lambda},
				\end{equation*}

			\item\label{it:splines3} for all $p\in (1,\infty]$ there exists a constant
				$C_{p,k}$  depending only on $p$ and
				$k$ such that for all $k$-martingale
				spline sequences
				$(f_n)$,
				\begin{equation*}
					\big\| \sup_n |f_n| \big\|_{L^p}
					\leq C_{p,k}
					\sup_n\|f_n\|_{L^p},\ 
				\end{equation*}
			\item\label{it:splines4}

				if $(f_n)$ is an
				$L^1$-bounded $k$-martingale spline sequence, then
				$(f_n)$ converges
				 almost surely to some $L^1$-function.

			\item	for all $p\in(1,\infty)$ and all positive 
				integers $k$, scalar-valued $k$-spline-differences
				converge unconditionally in $L^p$, i.e. for all
				$f\in L^p$,
				\begin{equation*}
					\big\|\sum_n \pm (P_n^{(k)} -
					P_{n-1}^{(k)})f \big\|_{L^p}
					\leq C_{p,k} \|f\|_{L^p},
				\end{equation*}
				for some constant $C_{p,k}$ depending only on
				$p$ and $k$.
		\end{enumerate}
	Note that	the $L^1$-boundedness stated in property (i) and $L^\infty$-boundedness of the
		operators $P_n^{(k)}$ are equivalent, as $P_n^{(k)}$ being an
		orthogonal projection operator is also  
	self-adjoint.
	\eqref{it:splines1} is proved in \cite{Shadrin2001},
	vector valued versions of \eqref{it:splines2}--\eqref{it:splines4}
	are proved  in \cite{PassenbrunnerShadrin2014,
				MuellerPassenbrunner2017} and (v) is proved in
				\cite{Passenbrunner2014}.
 The basic starting point in proving the results (ii)--(v) independently
of the filtration $(\mathcal F_n)$ is that (i) is true  and in this paper we
prove the analogue of (i) for Chebyshev spline spaces (see Section
\ref{sec:prel} for exact definitions and properties of Chebyshev splines):

 \begin{thm}\label{thm:general}
	 There exists a  finite positive constant $C$,  depending only on $\mathcal U_k$, 
	 so that for all partitions $\Delta$ of the
	unit interval $[0,1]$, the orthogonal projection operator $P_\Delta$ onto 
	the space of Chebyshev splines $S(\mathcal U_k;\Delta)$ is bounded on
	$L^\infty$ by the constant $C$, i.e.,
	\[
		\| P_\Delta f\|_\infty \leq C\|f\|_\infty,\qquad f\in
		L^\infty[0,1].
	\]
\end{thm}

 The organization of this article is as follows:
In Section 2, we recall basic definitions and facts involving Chebyshev spline
functions. Next, in Section 3, we prove Theorem \ref{thm:general} 'at the
boundary',
meaning that we prove $|P_\Delta f(0)|$ and $|P_\Delta f(1)|$ are bounded by
$C\|f\|_\infty$ with a constant $C$ independent of $\Delta$, provided $\Delta$
satisfies some additional conditions.
Finally, in Section 4, we prove Theorem \ref{thm:general} by reducing the
general case to the boundary and showing how the conditions on $\Delta$ from
Section 3 can be eliminated.

\section{Preliminaries}\label{sec:prel}
 As a basic reference to Chebyshev splines, we use the book
	\cite{Schumaker1981}, in particular Chapter~9.
Suppose that for a non-negative integer $k$, $w=(w_1,\ldots,w_k)$ is a vector consisting of $k$ positive
functions (weights) on $I=[a,b]$ with $w_i\in C^{k-i+1}$
for all
$i\in\{1,\ldots,k\}$. Then, let
$T f (x) = \int_a^x f(y)\dif y$ be the operator of integration, and $m_g f(x) = g(x) f(x)$  the
multiplication operator by the function $g$ and
define  the vector $u(w) = (u_1,\ldots,u_k)$ of functions by 
\begin{align*}
	u_i &= \Big(\prod_{\ell=1}^{i-1} m_{w_\ell} T\Big) w_i =
	w_1(\cdot)\int_a^\cdot
	w_2(s_2)\cdots \int_a^{s_{i-1}} w_{i}(s_{i}) \dif s_{i} \cdots \dif s_2,
\end{align*}
for all $i=1,\ldots,k$.
Let $t_1\leq t_2\leq \cdots \leq t_k$ be an increasing sequence of real numbers.
We set
\begin{equation*}
	d_i := \max \{ 0\leq j\leq k-1 : t_i = \cdots = t_{i-j} \}
\end{equation*}
and define the expression
\begin{equation}\label{eq:defD}
	D\begin{pmatrix}
		t_1,\ldots,t_k \\
		u_1,\ldots,u_k\end{pmatrix} := \det \big( D^{d_i} u_j (t_i)
	\big)_{i,j=1}^k,
\end{equation}
where $D$ denotes the ordinary differential operator.
The functions $u(w)$ form an \emph{extended complete
Chebyshev} system (ECT system for short), i.e., for all $j=1,\ldots,k$ and all
points $t_1\leq t_2 \leq \cdots \leq t_j$, we have that
\begin{equation}\label{eq: ECT>0}
	D\begin{pmatrix}
		t_1,\ldots,t_j \\
		u_1,\ldots,u_j\end{pmatrix}>0. 
\end{equation}
An example  of an ECT is  the set of monomials $\{1,x,\ldots,x^{k-1}/(k-1)!\}$, which corresponds to the choice $a=0$ and $w_1 = \cdots = w_k = 1$. 
In working with ECT systems, we conveniently define the corresponding
differentiation operators
\begin{equation*}
	D_i f = D\Big( \frac{f}{w_i} \Big),\qquad i=1,\ldots,k.
\end{equation*}
Additionally, put
\begin{equation*}
	L_i = D_i D_{i-1} \cdots D_1, \qquad i=1,\ldots,k
\end{equation*}
and $L_0 f = f$.
If $U_k = \{u_1,\ldots, u_k\}$ is an ECT, then $\mathcal U_k := \lin U_k$ is the null space of
the operator $L_k$, and, by definition of $u_i$,
\begin{equation*}
	L_j u_i(a) = w_i(a) \delta_{j,i-1},\qquad j=0,\ldots,i-1, \quad   i=1,\ldots,k.
\end{equation*}
We also set
\begin{equation}\label{eq:detL}
	D_{(L_0,\ldots, L_{k-1})}\begin{pmatrix}
	t_1,\ldots,t_k \\
		u_1,\ldots,u_k\end{pmatrix} := \det \big( L_{d_i} u_j (t_i)
	\big)_{i,j=1}^k.
\end{equation}
 In this article, we write  $A(t)\lesssim B(t)$ to mean that there exists a
constant $C$ that depends only on $\min_{x\in I} w_i(x)$ and on 
$\max_{x\in I} |D^j w_i(x)|$ for
$1\leq i\leq k$ and $0\leq
j\leq k-i+1$ so that $A(t)\leq CB(t)$ for all $t$, where $t$ denotes all
implicit and explicit dependencies that the objects $A$ and $B$ might have.
Similarly, we use the notations $\gtrsim$ and $\simeq$.

The determinant \eqref{eq:detL} and its derivatives can be compared to the classical Vandermonde determinant
$V(t_1,\ldots,t_k)$ 
that results from inserting the functions $u_j = t^{j-1}$ into \eqref{eq:defD}.
Indeed, 
 \begin{equation}\label{eq:vandermonde}
	\begin{aligned}
	\Big|L_j D_{(L_0,\ldots,
		L_{k-1})}&\begin{pmatrix}
		t_1,\ldots,t_{k-1},x \\
		u_1,\ldots,u_k\end{pmatrix}\Big| 
	\simeq |D^j
	V(t_1,\ldots,t_{k-1},x)|,
\end{aligned}
\end{equation}
where the differentiation operators $L_j$ and $D$ act with respect to the
variable $x$.
Also, for $u\in\mathcal U_k$, we have the following Markov inequality for all $1\leq p,q\leq
\infty$:
\begin{equation}\label{eq:markov}
	\|D^j u\|_{L^p(I)} \lesssim |I|^{-j+1/p-1/q}\|u\|_{L^q(I)},
\end{equation}
where the implied constant additionally may depend on $p$ and $q$.
As for algebraic polynomials, any non-zero function $u\in \mathcal U_k$ has at
most $k-1$ zeros including multiplicities.
We can associate a divided difference to the system $u(w)$
\begin{equation*}
	[t_1, \ldots, t_{k}]_{u(w)} f = \frac{	D\begin{pmatrix}
		t_1,\ldots,t_{k} \\
		u_1,\ldots,u_{k-1},f\end{pmatrix}}{	D\begin{pmatrix}
		t_1,\ldots,t_{k} \\
		u_1,\ldots,u_{k}\end{pmatrix}}
\end{equation*}
and observe that $[t_1,\ldots,t_{k}] u = 0$ for all
$u\in\lin\{u_1,\ldots,u_{k-1}\}$ and $[t_1,\ldots,t_{k}] u_{k}=1$.
If $u(w) = (u_1,\ldots,u_k)$, we define the truncated vector $\bar u (w) =
(u_1,\ldots,u_{k-1})$.
Similarly to the classical divided differences,  the following recursion formula 
for $k\geq 2$, $t_1 \neq t_k$ and sufficiently smooth
functions $f$ is true:
\begin{equation*}
	[t_1,\ldots, t_{k}]_{u(w)} f = \frac{[t_2,\ldots,t_{k}]_{\bar u(w)} f -
	[t_1,\ldots,t_{k-1}]_{\bar u(w)}f}{[t_2,\ldots,t_{k}]_{\bar u(w)}u_{k} -
	[t_1,\ldots,t_{k-1}]_{\bar u(w)} u_{k}}.
\end{equation*}

Given the weights $w=(w_1,\ldots,w_k)$ and the corresponding ECT-system $u(w)=(u_1,\ldots,
u_k)$, we define the dual canonical ECT-system $u^*(w) =
(u_1^*,\ldots,u_k^*,u_{k+1}^*)$ by
\begin{equation}\label{eq:dualU}
	u_i^* =  \Big(\prod_{\ell=1}^{i-1} T m_{w_{k-\ell+1}}\Big)
	\charfun_{I},\qquad i = 1,\ldots, k+1,
\end{equation}
and the associated operators $D_i^* f = (Df)/w_{i}$ for $i=1,\ldots,k$ and
$L_i^* = D_{k-i+1}^* \cdots D_k^*$ for $i=1,\ldots,k$ and $L_0^*f = f$.
%


We denote by $T_y f = \int_y^x f(s)\dif s$ and set
\begin{align*}
	h_j(x,y) = h_j^w(x,y) := \Big(\prod_{\ell=1}^{j-1} m_{w_\ell}
	T_y \Big)w_j(x),\quad 	g_j(x,y)=g_j^w(x,y) :=\charfun_{x\geq y}(x,y) h_j^w(x,y).
\end{align*}
The functions $g_j$ are the analogues of the truncated power functions $(x-y)_+^{j-1}$ for polynomials.
\subsection{Chebyshevian spline functions}\label{sec:chebsplines}
Let $\Delta= \{a = t_0=\cdots =t_{k-1}<t_k<\cdots< t_{n}< t_{n+1} = \cdots =
t_{n+k}=b\}$ and define
the Chebyshevian spline space $S(\mathcal U_k;\Delta)$ as 
\begin{align*}
	S(\mathcal U_k;\Delta) &= \{ s\in C^{k-2}(I): \text{there exist functions }  \\
	&\qquad s_{k-1},\ldots, s_n\in\mathcal U_k\text{ so that }
	s|_{(t_j,t_{j+1})} = s_j\text{ for all }j = k-1,\ldots,n\},
\end{align*}
where $\mathcal U_k = \lin U_k = \lin \{u_1,\ldots,u_k\}$.
Denote by $|\Delta| = \max_{i= 0}^n (t_{i+1} - t_i)$   the maximal
grid size of $\Delta$.
We define the Chebyshevian B-spline function $M_i^w$ for weights $w=
(w_1,\ldots,w_k)$ and for $x\in[a,b)$ by
\begin{equation*}
	M_i(x) = M_i^{w} (x) = (-1)^k [t_i,\ldots, t_{i+k}]_{u^*(w)} g_k(x,y),\qquad
	i=0,\ldots, n,
\end{equation*}
where the divided difference is taken with respect to the variable $y$. We
define $M_i(b)$ by extending $M_i$ continuously.
The system of functions $(M_i)$ forms an
algebraic basis of the space $S(\mathcal U_{k};\Delta)$ and we have
\begin{equation*}
	M_i>0 \text{ on } (t_i,t_{i+k}),\qquad M_i=0 \text{ on }
	(t_i,t_{i+k})^c.
\end{equation*}
We also have the Peano representation
\begin{equation*}
	[t_i,\ldots,t_{i+k}]_{u^*(w)} f = \int_{t_i}^{t_{i+k}} M_i(x)
	L_k^*f(x)\dif x
\end{equation*}
for sufficiently smooth functions $f$.
Setting $f=u_{k+1}^*$, we obtain that $M_i$ is $L^1$-normalized:
\begin{equation*}
	\int_{t_i}^{t_{i+k}} M_i(x)\dif x = 1.
\end{equation*}

Similar to classical polynomial splines, a nonzero function in $S(\mathcal U_k;\Delta)$ has at most
$\operatorname{dim}S(\mathcal U_k;\Delta)-1= n$ zeros.

Let $K:\Omega\times \Sigma \to \mathbb R$ be defined on two totally ordered sets
$\Omega,\Sigma$. Then $K$ is called totally positive if for any integer $r$ and any
choice $\omega_1\leq
\cdots \leq\omega_r$ of elements in $\Omega$ and $\sigma_1\leq \cdots \leq \sigma_r$ of
elements in $\Sigma$, the determinant of the matrix
$\big(K(\omega_i,\sigma_j)\big)_{i,j=1}^r$ is non-negative.
We note that the kernel $K(i,t)= M_i(t)$ is  totally positive
(see e.g. \cite[Theorem 9.34]{Schumaker1981} or
\cite[pp. 527]{Karlin1968}).
Using the basic composition formula for determinants (see e.g. \cite[p.
17]{Karlin1968}), we also get that the kernel $K(i,j)=\langle M_i, M_j\rangle$
is totally positive.

Additionally, we are able to perform Hermite interpolation; the simplest form is
the following: assume that $y_0 < \cdots < y_n$ are points that satisfy $y_i\in
(t_i,t_{i+k})$. Then, for all vectors $(v_j)_{j=0}^n$, there exists a unique function $s\in S(\mathcal
U_k;\Delta)$ so that 
\begin{equation}
	\label{eq:hermite}
	s(y_i) = v_i,\qquad i=0,\ldots,n.
\end{equation}

\subsection{Renormalized B-spline functions}
In the following, let $\phi_{i,r}$ be the function contained in
$\lin\{u_1^*,\ldots,u_r^*\}$ with zeros at the points $t_i,\ldots,t_{i+r-2}$
(including multiplicities) and  having 
leading coefficient $1$, i.e., $\phi_{i,r}$ is of the form
\begin{equation*}
	\phi_{i,r} = u_r^* + \sum_{j=1}^{r-1} c_j u_j^*
\end{equation*}
for some numbers $(c_j)_{j=1}^{r-1}$.
 Observe that, if $s$ is not contained in
 $\{t_{i},\ldots,t_{i+r-2}\}$,  
\begin{equation}\label{eq:phi}
	\phi_{i,r}(s) = \frac{ D\begin{pmatrix}
		t_i,\ldots,t_{i+r-2}, s \\
		u_1^*,\ldots,u_{r-1}^*,u_{r}^*\end{pmatrix} } { D\begin{pmatrix}
		t_1,\ldots,t_{i+r-2} \\
		u_1^*,\ldots,u_{r-1}^*\end{pmatrix} } = \frac{
		D_{(L_0^*,\ldots,L_{r-1}^*)}\begin{pmatrix}
		t_i,\ldots,t_{i+r-2}, s \\
		u_1^*,\ldots,u_{r-1}^*,u_{r}^*\end{pmatrix} } {
		D_{(L_0^*,\ldots,L_{r-2}^*)}\begin{pmatrix}
		t_1,\ldots,t_{i+r-2} \\
		u_1^*,\ldots,u_{r-1}^*\end{pmatrix} },
\end{equation}
and the difference $\phi_{i,k+1} - \phi_{i+1,k+1}$ is contained in
$\lin\{u_1^*,\ldots,u_k^*\}$ and has zeros at $t_{i+1},\ldots,t_{i+k-1}$
including multiplicities.
Therefore, there exists a number $\alpha_i$ so that 
\begin{equation}\label{eq:phi_relation}
	\phi_{i,k+1} - \phi_{i+1,k+1} = \alpha_i \phi_{i+1,k}.
\end{equation}
 We now verify that $\alpha_i$ is positive. 
If $\mu$ denotes the multiplicity of $t_{i+k}$ in the set
$\{t_{i+1},\ldots,t_{i+k-1}\}$, then
\begin{align*}
	D^j \phi_{i+1,k+1}(t_{i+k}) &= 0,\qquad 0\leq j\leq \mu, \\
	D^j \phi_{i,k+1}(t_{i+k}) &= D^j \phi_{i+1,k}(t_{i+k}) = 0,\qquad 0\leq
	j\leq \mu-1.
\end{align*}
We additionally have that 
\begin{equation*}
	D^\mu \phi_{i,k+1}(t_{i+k}) \neq 0,\qquad \text{and}\qquad D^\mu
	\phi_{i+1,k}(t_{i+k}) \neq 0,
\end{equation*}
for otherwise this, together with the other conditions imposed on $\phi_{i,k+1}$
and $\phi_{i+1,k}$ would imply $\phi_{i,k+1}\equiv 0 $ or $\phi_{i+1,k}\equiv 0$
which is impossible. Now we also know for $s> t_{i+k}$ that $\phi_{i,k+1}(s)>0$
and $\phi_{i+1,k}(s)>0$ by formula \eqref{eq:phi} and the fact that
$\{u_1^*,\ldots,u_{k+1}^*\}$ is an ECT system. Thus, by Taylor expansion,
\begin{equation*}
	D^\mu \phi_{i,k+1}(t_{i+k})>0\qquad \text{and}\qquad
	D^\mu\phi_{i+1,k}(t_{i+k})>0,
\end{equation*}
or equivalently
\begin{equation*}
	L_\mu^* \phi_{i,k+1}(t_{i+k})>0\qquad \text{and}\qquad
	L_\mu^*\phi_{i+1,k}(t_{i+k})>0,
\end{equation*}
which yields $\alpha_i>0$ after inserting into formula \eqref{eq:phi_relation}:
\begin{equation}\label{eq:phiDmu}
	L_\mu^* \phi_{i,k+1}(t_{i+k}) = \alpha_i L_\mu^* \phi_{i+1,k}(t_{i+k}).
\end{equation}
 Equation \eqref{eq:phi} and the comparison to Vandermonde determinants
\eqref{eq:vandermonde}  imply that 
 \begin{equation}\label{eq:alphasupp}
 \alpha_i \simeq t_{i+k} - t_i.
\end{equation}
If we define the renormalized B-spline function
\begin{equation}\label{eq:defN}
	N_i(x) = \alpha_i M_i(x),
\end{equation}
then the collection of those functions forms a partition of unity in the sense
%
\begin{equation}\label{eq:partunity}
		\sum_{i} N_i(x) = u_1(x).
	\end{equation}

\subsection{Dual functionals to $N_i$}
Let $J_i=[t_\ell, t_{\ell+1}]$ be a 
 largest grid interval contained in the support
$[t_i,t_{i+k}]$ of $N_i$. Let $g$ be a smooth function on the real line 
with $g=0$ on
$(-\infty,0]$ and $g=1$ on $[1,\infty)$.
We define the functionals
\begin{equation*}
	\lambda_i f = \frac{1}{\alpha_i} \int_{t_i}^{t_{i+k}} f(x) L_k^*
	\psi_i(x)\dif x,
\end{equation*}
where $\psi_i(x) = \alpha_i \phi_{i+1,k}(x) G_i(x)$ 
and $G_i(x) := g(\frac{x-\inf
J_i}{|J_i|})$.
Those functionals are dual to the B-spline functions $N_i$ in the sense that
$\lambda_i N_j = \delta_{i,j}$ and,  for $p\in[1,\infty]$, they satisfy the inequality
\begin{equation}\label{eq:dual}
		|\lambda_i f| \lesssim C|J_i|^{-1/p}
	\|f\|_{L^p(J_i)},
	\end{equation}
	provided that $|\Delta| \leq 1$  and where the implied constants
	additionally depend on $p$. 
	 An important consequence is the following stability of the functions $(N_i)$ under the
condition that $|\Delta| \leq 1$: 
%
for all $p\in[1,\infty]$,
	\begin{equation*}
		\| \sum_{i} c_i N_i \|_{p} \simeq \Big(\sum_i |c_i|^p
		\alpha_i\Big)^{1/p},
	\end{equation*}
	where, again, the implied constants may 
	additionally depend on the parameter $p$.
	The proof of this result proceeds in the same way as  for
	polynomial splines (see e.g. \cite[p. 145]{DeVoreLorentz1993}).


	 \section{Boundary points}
   
The main result of  this section is the following theorem:

\begin{thm}\label{thm:endpoint}
	There exist constants $\varepsilon, K_1>0$, 
	 depending only on $\mathcal U_k$, 
		so that for all partitions $\Delta$ of the interval $[a,b]$
		satisfying  $|\Delta| \leq
	\varepsilon$, the orthogonal projection operator $P_\Delta$ onto the
	space of Chebyshev splines $S(\mathcal U_k;\Delta)$ satisfies
	\begin{equation*}
		|P_\Delta f(a)| + |P_\Delta f(b)|\leq
			K_1\|f\|_\infty,\qquad f\in L^\infty[a,b].
		\end{equation*}
	 \end{thm}

 Similarly as it is done in 
\cite{Shadrin2001} and \cite{Golitschek2014} for polynomial splines,
we  
construct a function $\phi\in S(\mathcal U_k;\Delta)$ with the
properties that the sign of $\langle \phi,M_j\rangle$ is alternating in $j$ and
that, for all $j$, $|\langle \phi,M_j\rangle| \gtrsim 1$. The basic idea of the
construction of such a function $\phi$ is the same as in \cite{Shadrin2001},
but we have to deal with the fact that the weights $(w_i)$ now are arbitrary
functions.
In Section 3.1, we give a few estimates and formulas on the B-spline functions
$(M_i)$ that are needed subsequently. In Section 3.2, the function $\phi$ is
defined and it is shown that it has the desired properties and in Section 3.3,
we show how those properties imply Theorem \ref{thm:endpoint}.
 

\subsection{Estimates for the derivatives of  B-spline functions}
The following formula for the derivative of the B-spline functions $M_i$ in
terms of B-splines of lower order is a consequence of the recursion formula for
divided differences.
\begin{lem}\label{lem:M deriv}Let $w=(w_1,\ldots,w_k)$, $\bar w= (w_2,\ldots,w_k)$  and $u^*(w) =
	(u^*_1,\ldots u^*_{k+1})$ be the dual system to $u(w)$.
	Then, for any $x\notin \{t_i, \ldots, t_{i+k}\}$,
	\begin{equation}\label{eq:M deriv}
			D_1 M_i^w(x)=
			 \left(\frac{M_i^w(x) }{w_1(x) }\right)' =
			 \frac{M_i^{\bar w}(x)-M_{i+1}^{\bar
		w}(x)}{h^w_i},
	\end{equation}
	where
	$$h^w_i=[t_{i+1},\ldots,t_{i+k}]_{u^*(\bar w)} u_{k+1}^*
		-[t_{i},\ldots,t_{i+k-1}]_{u^*(\bar w)} u_{k+1}^* .$$
	Additionally, 
	\begin{equation}\label{eq:est_h}
		0 < h_i^w \lesssim t_{i+k} - t_i.
	\end{equation}
\end{lem}
\begin{proof}

	We use the definitions of $M_i^w$ and of $g_k(x,y)$ to get
\begin{align*}
	M_i^w(x) &= (-1)^k [t_i,\ldots,t_{i+k}]_{u^*(w)} g_k^w (x,y) \\
	&=(-1)^k [t_i,\ldots,t_{i+k}]_{u^*(w)} \Big(\charfun_{x\geq y} (x,y)
	\big(\prod_{\ell=1}^{k-1} m_{w_\ell} T_y\big)w_k(x) \Big),
\end{align*}
where the divided difference is applied to the $y$-variable. If we divide by
$w_1$ and differentiate, we obtain
\begin{align*}
	D_1 M_i^w(x) &= (-1)^k [t_i,\ldots,t_{i+k}]_{u^*(w)} \Big( \charfun_{x\geq y}
	(x,y) \big(\prod_{\ell=2}^{k-1} m_{w_\ell} T_y\big) w_k(x) \Big)\\
	&= (-1)^k[t_i,\ldots,t_{i+k}]_{u^*(w)} g_{k-1}^{\bar w}(x,y).
\end{align*}
Now, we can use the recursion formula for divided differences to get for $\bar
u^*(w) = (u_1^*,\ldots,u_k^*)$
\begin{equation*}
	D_1 M_i^w (x)= (-1)^k \frac{[t_{i+1},\ldots,t_{i+k}]_{\bar
	u^*(w)}-[t_{i},\ldots,t_{i+k-1}]_{\bar u^*(w)}}{[t_{i+1},\ldots,t_{i+k}]_{\bar
	u^*(w)}u_{k+1}^*-[t_{i},\ldots,t_{i+k-1}]_{\bar u^*(w)}u_{k+1}^*}
	g_{k-1}^{\bar w} (x,y).
\end{equation*}
Since $u^*(\bar w) = \bar u^*(w)$ we obtain the desired formula for $D_1M_i^w$
by using the definition of $M_i^{\bar w}$.

 Next, we show $h_i^w>0$ and first consider the case
	$t_{i+k} > t_{i+k-1}$. Here, for $x\in (t_{i+k-1}, t_{i+k})$ we have by the
above formula
\[
	\frac{M_i(x)}{w_1(x)} = \frac{1}{h_i^w} \int_{t_i}^x \left(M_i^{\bar w}(s) -
	M_{i+1}^{\bar w}(s)\right)\dif s = \frac{1}{h_i^w}\Big( 1 - \int_{t_i}^x
	M_{i+1}^{\bar w}(s)\dif s\Big).
\]
Since the left hand side is positive and the term in the   brackets on the
right hand side is positive, we conclude that $h_i^w$ is positive  as well.

 It remains to consider $t_{i+k} = t_{i+k-1}$, in which case we define the
function
\begin{equation*}
	\begin{aligned}
	g(t) = 
	D&\begin{pmatrix} 
		t_{i+1} & \cdots & t_{i+k-1} & t \\
		u_1^* & \cdots & u_{k-1}^* & u_{k+1}^* \\			       
		\end{pmatrix}\cdot
	D\begin{pmatrix}
		t_{i} & \cdots & t_{i+k-1} \\
		 u_1^* & \cdots & u_{k}^* \\
		  \end{pmatrix} \\
	&- D\begin{pmatrix} 
		t_{i} & \cdots & t_{i+k-2} & t_{i+k-1}  \\
		u_1^* & \cdots & u_{k-1}^* & u_{k+1}^* \\
		\end{pmatrix}\cdot
	D\begin{pmatrix}
		t_{i+1} & \cdots & t_{i+k-1} & t\\
		 u_1^* & \cdots & u_{k-1}^* &  u_{k}^* \\
		 \end{pmatrix}
\end{aligned}
\end{equation*}
for $t\notin \{t_i,\ldots, t_{i+k}\}$ and extend it smoothly to $\{t_i,\ldots,
	t_{i+k}\}$. 
Then, all the zeros of $g$ (including multiplicities) are 
the points $\{t_i,\ldots,t_{i+k-1} \}$ and in particular $D^j g(t_{i+k}) = 0$
for $j=0,\ldots,\mu-1$ where $\mu$ denotes the multiplicity of $t_{i+k}$ in the
set $\{t_i,\ldots, t_{i+k-1}\}$.
By the definition of $h_{i}^w$ and  \eqref{eq: ECT>0},  the sign of $h_{i}^w$ is the same as the sign
of $D^\mu g(t_{i+k})$. Since $g(t_{i+k} + \varepsilon)>0$ by the case
$t_{i+k}>t_{i+k-1}$, we also deduce $D^\mu g(t_{i+k})>0$.

The upper estimate for $h_i^w$ now follows from
\begin{equation*}
	h_i^w	 = \int_{t_i}^{t_{i+k}} w_1(x) \Big(\int_{t_i}^x M_i^{\bar
	w}(s)\dif s -
	\int_{t_i}^x
	M_{i+1}^{\bar w}(s)\dif s\Big) \dif x,
\end{equation*}

where we  have used the normalization $\int M_i^w = 1$.
\end{proof}

	 Let us denote by $M_j^{(\ell)}$ the $j$th B-spline corresponding to the
	weights $(w_{k-\ell+1},\ldots,w_k)$ and denote by $h_j^{(\ell)}$ the
	corresponding factor in formula \eqref{eq:M deriv}, i.e. formula
	\eqref{eq:M deriv} reads
	\begin{equation}
		\label{eq:M deriv 1}
		D_1 M_j^{(\ell)} = \frac{M_j^{(\ell-1)} -
		M_{j+1}^{(\ell-1)}}{h_j^{(\ell)}}.
	\end{equation}

\begin{lem}\label{lem: deriv M est}
	We have $\sgn D_{k-1}\cdots D_1 M_j^{(k)}(t)=(-1)^i$ for $t\in
	(t_{j+i},t_{j+i+1})$ with $i=0,\ldots,k-1$, and  for any $\ell = 0,\ldots k-1$ and
	$t\in[t_j,t_{j+k}]\backslash\{t_j,\ldots,t_{j+k}\}$,
	\begin{equation*}
		|D_{k-1}\cdots D_1 M_j^{(k)}(t)| \gtrsim (t_{j+k}-t_j)^{-k}
		\cdot\max\big(1, (t_{j+k}-t_j)^{\ell+1} |D_\ell \cdots D_1
		M_j^{(k)}(t)|\big).
	\end{equation*}
%
\end{lem}
\begin{proof}
	Let us first observe that by \eqref{eq:M deriv 1}, for $\ell = 0,\ldots, k-1$, we can write $D_\ell\cdots
D_1M_j^{(k)}$ in terms of the functions $(M_m^{(k-\ell)})$ as
\[
	D_\ell \cdots D_1 M_j^{(k)}(t) = \sum_{m} \beta_m^{(\ell)}
	M_m^{(k-\ell)}(t),
\]
where the coefficients $\beta_m^{(\ell)}$ are given by the following recursion
formulas:
\begin{equation}\label{eq:rec_coeff}
	\begin{aligned}
		\beta_j^{(0)}&=1,\qquad
		\beta_{m}^{(\ell)}=0 \qquad \text{for }m\notin \{j,\ldots,j+\ell\} \\
		\beta_m^{(\ell)} &=
		\frac{\beta_m^{(\ell-1)}}{h_m^{(k-\ell+1)}} -
		\frac{\beta_{m-1}^{(\ell-1)}}{h_{m-1}^{(k-\ell+1)}},\qquad
		m=j,\ldots,j+\ell. 
	\end{aligned}
\end{equation}
Using \eqref{eq:est_h}, this implies in particular that $\sgn \beta_{m}^{(\ell)} = (-1)^{m-j}$ and 
\begin{equation}
	\label{eq:beta}
	|\beta_m^{(k-1)}|\gtrsim (t_{j+k}-t_j)^{-(k-\ell)+1} \sum_{r=m-(k-\ell)+1}^m
|\beta_r^{(\ell)}|\qquad \text{for }m=j,\ldots,j+\ell.
\end{equation}

Taking a point $t\in (t_m,t_{m+1})$ for $m= j,\ldots,j+k-1$ and choosing $\ell = 0$ in formula \eqref{eq:beta}, yields 
\begin{align*}
	|D_{k-1} \cdots D_1 M_j^{(k)}(t)| &= \Big|\sum_{r}
	\beta_r^{(k-1)} \cdot M_r^{(1)}(t)\Big|  = |\beta_m^{(k-1)}|\cdot
	M_m^{(1)}(t)\\
	&\gtrsim (t_{j+k}-t_j)^{-k+1}
		\sum_{r=m-k+1}^m |\beta_r^{(0)}| (t_{m+1}-t_m)^{-1} \\
	&\gtrsim (t_{j+k}-t_j)^{-k},
\end{align*}
which shows the first part of the desired inequality.

On the other hand,
we can compute for $t\in (t_m,t_{m+1})$,  by the support property $\supp
	M_r^{(k-\ell)} = [t_r, t_{r+k-\ell}]$ of the B-spline functions,
\begin{align*}
	|D_\ell\cdots D_1 M_j^{(k)}(t)| = \Big| \sum_{r=m-(k-\ell)+1}^m \beta_r^{(\ell)}
	M_r^{(k-\ell)}(t)\Big|.
\end{align*}
As a consequence of \eqref{eq:alphasupp},\eqref{eq:defN},\eqref{eq:partunity},
we get that $|M_r^{(k-\ell)}(t)| \lesssim (t_{r+k-\ell} - t_r)^{-1}$, which,
together with \eqref{eq:beta},
implies
\begin{align*}
	|D_\ell\cdots D_1 M_j^{(k)}(t)| &\lesssim (t_{m+1} -
	t_m)^{-1}\sum_{r=m-(k-\ell)+1}^m |\beta_r^{(\ell)}| \\
	&\lesssim (t_{m+1} - t_m)^{-1}(t_{j+k}-t_j)^{k-\ell-1} |\beta_m^{(k-1)}| \\
	&\lesssim (t_{j+k}-t_j)^{k-\ell-1} |D_{k-1}\cdots D_1 M_j^{(k)}(t)|,
\end{align*}
which  also shows the second part of the desired inequality.
\end{proof}

\subsection{Definition and Properties of the Chebyshev spline $\phi$}
Let  $\Delta= \{a = t_{-k+1}=\cdots =t_{k-1}<t_k<\cdots< t_{n}< t_{n+1} = \cdots =
t_{n+2k-1}=b\}$ and  $\sigma$ be the Chebyshev spline  on the grid $\Delta$ so
that in
the interior of each grid interval,
\begin{equation}\label{eq:defsigma}
	\sigma \in \ker D_kD_{k-1}\cdots D_2 D_2\cdots D_{k}D_{k+1},
\end{equation}
(where   $D_j = D\circ m_{1/w_j}$ and $D_{k+1}=D$)
with the properties
\begin{enumerate}
	\item \label{it:sigma-a}$D_3\cdots D_{k}D_{k+1}\sigma(a) = 1$,\qquad
	\item $D_i\cdots D_{k}D_{k+1}\sigma(a) =D_i\cdots D_{k}D_{k+1}\sigma(b) =  0$ for any $i=4,\ldots, k+1$.
	\item $\sigma(t_j) = 0$ for any $j=k-1,\ldots n+1$.
\end{enumerate}
 By Hermite interpolation \eqref{eq:hermite} such a function $\sigma$ exists and is
uniquely determined.
 Indeed, denote by $(M_j^{(2k-1)})_{j=-k+1}^n$ the B-splines with respect
to the partition $\Delta$ corresponding to the
differential operator in \eqref{eq:defsigma}. Then, conditions (2)  and
$\sigma(a)=\sigma(b)=0$ for $\sigma=\sum_{j=-k+1}^n a_j
M_j^{(2k-1)}$ amount to solving two triangular systems resulting in  $a_{-k+1} = \cdots = a_{-2} = 0$ and $a_n=
\cdots =
a_{n-k+3}=0$, since $D_i\cdots D_{k+1} M_j^{(2k-1)}(a)$ is only non-zero for
$j=-k+1,\ldots, -i+2$ and $D_i\cdots D_{k+1} M_j^{(2k-1)}(b)$ is only non-zero for 
$j=n+i-k-1,\ldots,n$.
Next, we choose $a_{-1}$ so that $D_3\cdots D_k D_{k+1}
\sigma(a)=1$.
Then, we can use standard Hermite interpolation \eqref{eq:hermite} to obtain
uniquely determined
coefficients
$(a_{j})_{j=0}^{n-k+2}$ satisfying $\sigma(t_k)= \cdots= \sigma(t_n)=0$.

%

Next, we define  the operators
$$S_i:=\begin{cases}
D_{k+2-i}\cdots D_{k+1}, & { \rm  if }\ i\leq k\\
D_{i-k+1}\cdots D_2 D_2\cdots D_{k+1},& { \rm  if }\ k+1\leq i\leq 2k-1
\end{cases}$$
 and the function $\phi$ by using $\sigma$ and
the  operators $S_i$ as follows:
$$\phi := \frac{w_1}{w_2}S_{k-1} \sigma.$$ Note that $D_k\cdots D_1\phi=D_k\cdots D_2D_2 S_{k-1} \sigma=0$
 on each interval $(t_i,t_{i+1})$.

We now observe that $\sigma$ has $n+k-1$ zeros and this is the
highest number of zeros a nonzero spline of this order can have  (including
multiplicities), so we conclude that  $\sgn \sigma= (-1)^{j-k+1}$ on
$(t_j,t_{j+1})$ for all indices $j$, where we also used the fact that $\sigma$
is positive on $(t_{k-1},t_k)$ by conditions (1), (2) and (3). 

\begin{lem}\label{lem:sigma k}
	If  $|\Delta| \leq 1$, then	
	 
	\begin{equation*}
		\int_{t_j}^{t_{j+1}} |\sigma(t)|\dif t \gtrsim  (t_{j+1} -
		t_j)^{k},\qquad j=k-1,\ldots,n.
	\end{equation*}
\end{lem}
\begin{proof}
We introduce the function
\begin{equation*}
	H = \left(\frac{S_{k-1}\sigma}{w_2}\right)^2 + 2\sum_{q=1}^{k-1}(-1)^q\frac{S_{k-1-q}\sigma}{ w_{q+2}}\cdot \frac{S_{k-1+q}\sigma}{w_{q+1}}.
\end{equation*}
Since all the functions $\sigma,S_1\sigma,\ldots, S_{2k-2}\sigma$ are
continuous at the grid points $t_j$ and
$\sigma(t_j)=0$, the function $H$ is continuous at the grid points. Moreover, $H$ is differentiable in the
interior $(t_j,t_{j+1})$ of the grid intervals. This
yields
\begin{align*}
	H'& = 2\frac{S_{k-1}\sigma}{w_2}S_{k}\sigma + 2\sum_{q=1}^{k-1}(-1)^q\left(S_{k-q}\sigma\cdot \frac{S_{k+q-1}\sigma}{w_{q+1}} + \frac{S_{k-q-1}\sigma}{ w_{q+2}}\cdot {S_{k+q}\sigma}  \right)\\
	&=2\sum_{q=2}^{k-1}(-1)^q S_{k-q}\sigma\cdot
	\frac{S_{k+q-1}\sigma}{w_{q+1}}+ 2\sum_{q=1}^{k-2}(-1)^q
	\frac{S_{k-q-1}\sigma}{ w_{q+2}}\cdot {S_{k+q}\sigma} \\
	 &\qquad+(-1)^{k-1} \frac{\sigma}{ w_{k+1}}\cdot {S_{2k-1}\sigma}	
	=(-1)^{k-1} \frac{\sigma}{ w_{k+1}}\cdot {S_{2k-1}\sigma}=0,
\end{align*}
where the last equality holds due to the condition  $\sigma\in\ker
D_{k}\cdots D_2 D_2\cdots D_{k+1}$. So we deduce that $H$ is constant on
all of the interval $[a,b]$ and $H(a)= 1/w_2^2(a)$.

By Markov's inequality \eqref{eq:markov}, we have that for all $r=0,\ldots, 2k-2$ and
$I_j=(t_j,t_{j+1})$,  
\[
	\|S_r \sigma\|_{L^\infty(I_j)} \lesssim |I_j|^{-r-1}\int_{I_j} |\sigma(t)|\dif t.
\]
 Thus, for all $t\in(t_j,t_{j+1})$ and all $j$,
\begin{equation*}
	\frac{1}{w_2(a)^2} = H(t) \lesssim |I_j|^{-2k} 
	\Big(\int_{I_j} |\sigma(t)|\dif t\Big)^2,
\end{equation*}
which is the conclusion.
\end{proof}

Now we are ready to prove the two desired properties of the Chebyshev spline $\phi$.
\begin{thm}\label{thm: phi_2}
	There exist two numbers $\varepsilon,c>0$ depending only on the weights
	$w_i$ so that if  $|\Delta| \leq \varepsilon$, we have  
		\begin{enumerate}
		\item $\sign\langle\phi,M_j\rangle$ is alternating,
		\item $|\langle \phi, M_j\rangle| \geq c$.
	\end{enumerate}
\end{thm}
\begin{proof}
 	We note that by the definition of $\phi$ and $\sigma$, we obtain by
	iterated partial integration that
	\begin{equation*}
		\int_{t_{j}}^{t_{j+k}} \phi(t) M_j(t) \dif t = (-1)^{k-1}
		\int_{t_j}^{t_{j+k}} \sigma(t) u(t)\dif t,
	\end{equation*}
	with $u= D_k \cdots D_2 m_{w_1} M_j$.
	Note that the function $u$ can be written in the form
	\[
		u= \sum_{\ell=0}^{k-1}\widetilde{w}_\ell\cdot D_\ell\ldots D_1
		M_j,
	\]
 where $\widetilde{w}_\ell$ are bounded in terms of the weight functions
 and $\widetilde{w}_{k-1}=\frac{w_{k-1}}{w_k}\cdot\ldots\cdot
 \frac{w_2}{w_3}\cdot  \frac{w_1^2}{w_2}=\frac{w_{1}^2}{w_k}$ additionally is
 bounded away from zero.
 Using Lemma \ref{lem: deriv M est} we obtain that there
 exists $\varepsilon>0$ depending only on the weight functions $w_1,\ldots,w_k$ so
 that, if  $|\Delta|\leq \varepsilon$ then 
\begin{equation}\label{eq: u}
\sgn u =\sgn D_{k-1}\ldots D_1 M_j \ \hbox{   and   }\ |u|\gtrsim |D_{k-1}\ldots D_1 M_j|.
\end{equation}
 
	We have $ \sgn\sigma(t)=(-1)^{j+i-k+1}$  for $t\in (t_{j+i},
	t_{j+i+1})$  and also, by using Lemma \ref{lem: deriv M est} we
	obtain $ \sgn D_{k-1}\ldots D_1M_j(t)=(-1)^{i}$  for $t\in (t_{j+i},
	t_{j+i+1})$. Hence, 
	\begin{align*}
		(-1)^{j}\langle\phi,M_j \rangle
&=\sum_{i=0}^{k-1}\int_{t_{j+i}}^{t_{j+i+1}}  (-1)^{j+i-k+1}\sigma(t) \cdot(-1)^{i} u(t)\dif t\\
& =\sum_{i=0}^{k-1}\int_{t_{j+i}}^{t_{j+i+1}} \left| \sigma(t) \cdot  u(t)\right|\dif t=
\int_{t_{j}}^{t_{j+k}}\left| \sigma(t) \cdot  u(t)\right|\dif t.
	\end{align*}
	If $J_j=(t_i,t_{i+1})$ is a largest subinterval of $[t_j,t_{j+k}]$, we
	obtain by  \eqref{eq: u}, Lemma \ref{lem:sigma k} and Lemma \ref{lem: deriv M
	est},
\begin{align*}
		(-1)^{j}\langle\phi,M_j \rangle
		& \gtrsim \int_{t_{i}}^{t_{i+1}} \left| \sigma(t)\right| \cdot
		\left| D_{k-1} \cdots D_1 M_j\right|\dif t  \gtrsim 1,
	\end{align*} 
	which shows both (1) and (2).
\end{proof}

\subsection{Proof of Theorem \ref{thm:endpoint}}
Now,  let  $K(\tau,t)$ be the Dirichlet kernel associated to the
orthogonal projection operator $P_\Delta$  onto $S(\mathcal U_k;\Delta)$, i.e.,
$P_\Delta$ is given by the formula
\begin{equation*}
	P_\Delta f(\tau) = \int_a^b K(\tau,t) f(t) \dif t.
\end{equation*}
Then, by duality,
 \begin{equation}\label{eq:dirichletkernel}
	\sup_{\|f\|_\infty\leq 1} |P_\Delta f(\tau)| = \int_a^b |K(\tau,t)|\dif
	t.
\end{equation}
We  will need the following result concerning the sign of the B-spline coefficients of
	$K(a,\cdot)$.
\begin{lem}\label{lem: ci alter}
	The coefficients $(c_i)$ of  
		$K(a,t) = \sum_{i=0}^n c_i M_i(t)$ 
	satisfy
	\begin{equation*}
		(-1)^{i} c_i \geq 0,\qquad i=0,\ldots,n.
	\end{equation*}
\end{lem}
\begin{proof}
	Observe that $c_i$ is given by
	
\begin{equation*}
	c_i = \sum_{j=0}^n \beta_{ij} M_j(a) = \beta_{i0}
\end{equation*}
where $(\beta_{ij})$  is the inverse to the Gram matrix $B=(\langle M_i,
M_j\rangle)$. Since $B$ is totally positive  (see Section
	\ref{sec:chebsplines}) and the inverse of a
totally positive matrix is checkerboard (i.e. alternates in sign), the lemma is
proved.
\end{proof}

Now let us turn to the proof of Theorem \ref{thm:endpoint}.
By definition of $\phi$, we have  $\phi\in S(\mathcal U_k;\Delta)$ 
and hence 
\[
	\frac{w_1(a)}{w_2(a)} = \phi(a) = P_\Delta\phi(a) = \int_a^b
	K(a,t)\phi(t)\dif t =
	\sum_i c_i \langle \phi,M_i \rangle.
\]
    If 
  $|\Delta|$ is sufficiently
small, then, collecting Theorem \ref{thm: phi_2}, Lemma \ref{lem: ci alter}, we can conclude
  \begin{align*}
	\frac{w_1(a)}{w_2(a)}&=\sum_i |c_i|\cdot |\langle\phi, M_i\rangle| 
	\geq c\cdot \sum_i |c_i| \\ 
	&\geq c\cdot\| \sum_i
	c_iM_i\|_1 = \int_a^b |K(a,t)|\dif t = \sup_{\|f\|_\infty\leq 1}|P_\Delta
	f(a)|. \qedhere
\end{align*}
 The estimate for $P_\Delta f(b)$ is proved similarly by exchanging
	condition (1) on page \pageref{it:sigma-a} for $\sigma$ by $D_3\cdots
	D_k D_{k+1}\sigma(b)=1$.

	 \section{From boundary points to the general case}
 In this section, we are going to prove Theorem \ref{thm:general}.
As an intermediate step, we show the following theorem, which still has
a restriction on the size of the 
partition $\Delta$ in its assumptions.
\begin{thm}\label{thm:generalwithepsilon}
	There exist constants $\varepsilon, K_2 >0$, only depending on  $\mathcal
	U_k$,
	so that if the partition $\Delta$ satisfies
	$|\Delta| \leq \varepsilon$, then the orthogonal projection operator
	$P_\Delta$ onto $S(\mathcal U_k;\Delta)$ satisfies
	\[
		\| P_\Delta f\|_\infty \leq K_2 \|f\|_\infty,\qquad f\in
		L^\infty.
	\]
\end{thm}
In the case of polynomial splines (i.e. the weights $w_1,\ldots, w_k$ each set
to be the constant function $1$), it was shown in \cite{Golitschek2014} that
Theorem \ref{thm:endpoint} (in the form for polynomial splines without the
condition on $|\Delta|$) implies this result. By analyzing the proof in
\cite{Golitschek2014}, we see that this transition can be obtained in a more
general setting which we now describe:

For a partition $\Delta = \{a=t_{k-1} < t_{k} < \cdots <t_{n} < t_{n+1}=b\}$ and for
 $x\in
(a,b)$, we let $m(x)$ be the unique index so that $x\in (t_{m(x)},
t_{m(x)+1}]$ and define the subpartition $\Delta_{x,\ell}=\{t_{k-1}< t_{k}< \cdots<
t_{m(x)}<x\}$ to the left of $x$ and the subpartition  $\Delta_{x,r}= \{x\leq
	t_{m(x)+1} < \cdots < t_{n+1}\}$ to the right of $x$.
To each partition $\Delta$, we associate a vector space $V(\Delta)$ of real
valued functions defined on $[a,b]$ and with it a basis
$(Q_\ell^\Delta)_{\ell=0}^{n}$ of $V(\Delta)$ consisting of functions with
local support $\supp Q_\ell^\Delta = [t_\ell,t_{\ell+k}]$, where we extend
$\Delta$ in the known way $t_{0} = \cdots = t_{k-2} = t_{k-1}$ and $t_{n+k} =
\cdots = t_{n+2} = t_{n+1}$.
Let $P_{V(\Delta)}$ denote the orthogonal projection operator with respect to
Lebesgue measure onto $V(\Delta)$.
Moreover we assume that $\mathcal C(\Delta)$ is a condition that depends on the
partition $\Delta$ and we require that the objects $V(\Delta)$ and $\mathcal
C(\Delta)$ are compatible in the following sense:

there exists a constant $C$ so that for all partitions $\Delta=\{t_{k-1} < \cdots <
t_{n+1}\}$
satisfying $\mathcal C(\Delta)$, the following conditions are satisfied:

%
\begin{enumerate}
	\item for all functions $f=\sum_i c_i Q_i^\Delta\in V(\Delta)$, 
		\begin{equation}\label{eq:localdeBoor}
			|c_\ell| \leq C \Big\|\sum_{i} c_i
			Q_i^\Delta \Big\|_{L^\infty(J_\ell)}, \qquad \ell =
			0,\ldots, n,
		\end{equation}
		where $J_\ell$ denotes the largest subinterval $[t_i,t_{i+1}]$
		of $\supp Q_\ell^\Delta = [t_\ell,t_{\ell+k}]$,
\item for any $x\in (a,b)$, the functions  $(Q_j^\Delta
		\charfun_{[a,x]})_{j=0}^{m(x)}$ form a basis of
	$V(\Delta_{x,\ell})$ and  the functions
	 $(Q_j^\Delta \charfun_{[x,b]})_{r=m(x)-k+1}^{n}$ form a basis of
	$V(\Delta_{x,r})$,
	\item for all $f\in V(\Delta)$, all indices $j$ and all subintervals $E$
	of $I_j=[t_j,t_{j+1}]$,
		 \begin{equation*}
			\| f \|_{L^\infty(I_j)} \leq C \frac{|I_j|^{k-1}}{|E|^{k}}
			\|f\|_{L^1(E)},
		\end{equation*}
	\item for all $x\in (a,b)$, $\mathcal C(\Delta_{x,\ell})$ and $\mathcal
		C(\Delta_{x,r})$ are both satisfied.
\end{enumerate}

Then, if we follow the arguments in \cite{Golitschek2014}, we have
the following theorem:
\begin{thm}\label{thm:goli}
	Assume that to each partition  $\Delta=\{a=t_{k-1}< \cdots <
	t_{n+1}=b\}$, we associate a function space
	$V(\Delta)$ so that the collection of all spaces $V(\Delta)$ with
	the conditions $\mathcal
	C(\Delta)$ have properties (1)--(4).
	Moreover, assume that there exists a constant $K_1$ so that for all partitions
	$\Delta$ satisfying
	$\mathcal  C(\Delta)$ we have
	 $\sup_{\|f\|_\infty \leq 1} \big(|P_{V(\Delta)} f(a)|+|P_{V(\Delta)}
	f(b)|\big)\leq K_1$.
	
	Then there exists a
	constant $K_2$ so that for all $\Delta$ satisfying $\mathcal C(\Delta)$,
	\begin{equation*}
		 \|P_{V(\Delta)} f\|_\infty \leq K_2\|f\|_\infty,\qquad f\in
		 L^\infty[a,b].
	\end{equation*}
\end{thm}

In order to prove Theorem \ref{thm:generalwithepsilon}, we apply Theorem
\ref{thm:goli} to Chebyshev splines as follows:
we let $\mathcal C(\Delta)$ be the condition
$|\Delta| \leq \varepsilon$ with the parameter $\varepsilon$ from Theorem
\ref{thm:endpoint}. The local basis $(Q_i^\Delta)$ will be the collection of the
B-spline functions $(N_i)$ and $V(\Delta)=S(\mathcal U_k;\Delta)$.
Those settings satisfy conditions (1)--(4). Indeed, (1) is a consequence of
\eqref{eq:dual}, (2) follows from the definition of $S(\mathcal
U_k;\Delta)$, (3) follows from the Markov inequality \eqref{eq:markov}  and
Taylor expansion of functions in $\mathcal U_k$ and (4)
is satisfied by definition of $\mathcal C(\Delta)$.
Thus, Theorem \ref{thm:generalwithepsilon} is a consequence of  Theorem
	\ref{thm:endpoint} and Theorem \ref{thm:goli}.

Next, we eliminate the restriction $|\Delta|\leq
	\varepsilon$ in Theorem \ref{thm:generalwithepsilon}. For this purpose,
	we will need the following geometric decay inequality of the inverse of
	the Gram matrix of B-spline functions:
	\begin{cor}\label{cor:geom}
		Suppose that $\Delta$ is such that $\|P_\Delta:L^\infty\to L^\infty\|\leq K$
		for some constant $K$. Then,
		there exist  two constants $C$ and $q<1$ depending only on $K$ so that 
		 
	\begin{equation*}
		|a_{ij}| \leq \frac{C q^{|i-j|}}{\alpha_i+\alpha_j},
	\end{equation*}
	where $(a_{ij})$ denotes the inverse of the matrix $(\langle N_i,
	N_j\rangle)$.
\end{cor}
This inequality is a consequence of Theorem
\ref{thm:generalwithepsilon} in the same way as its polynomial spline
counterpart in \cite{Ciesielski2000} (see also
	\cite{PassenbrunnerShadrin2014}) is a consequence of Shadrin's theorem.
Using this geometric decay inequality, we can prove the  main theorem: 

\begin{proof}[ Proof of Theorem \ref{thm:general}]
We assume that  $\tilde\Delta$ is a partition and $\Delta$ is a partition that we get
from adding one point to $\tilde\Delta$  in the middle of the largest grid  interval
$I$ in $\tilde{\Delta}$
with length $> \varepsilon$ and we assume that $\|P_\Delta:L^\infty\to L^\infty\|\leq K$.
Then, by definition of the orthogonal projections $P=P_\Delta$ and
 $\tilde{P}=P_{\tilde{\Delta}}$, we have,
if $(N_i)$ denotes the Chebyshevian B-spline basis of $S(\mathcal
U_k;\Delta)$,
\begin{equation*}
	\langle (\tilde P - P)f,N_i\rangle=0
\end{equation*}
if $|I\cap \supp N_i|=0$ 
since in this case $N_i$ is both in the range of $P$ and $\tilde{P}$ and thus
$\langle \tilde{P}f - f, N_i\rangle = \langle Pf - f, N_i\rangle = 0$.
Thus, we can expand
\begin{equation*}
	(\tilde{P} - P)f = \sum_{j:|\supp N_j\cap I|>0} c_j N_j^*
\end{equation*}
with $c_j=\langle (\tilde{P} - P)f,N_j\rangle$ and $(N_i^*)$ being the dual
basis to $(N_i)$ which is given by $N_j^*=\sum_i a_{ji} N_i$. Here, as in
Corollary \ref{cor:geom},  $(a_{ij})$
denotes the inverse of the matrix $(\langle N_i,N_j\rangle)$. Since $\tilde{P}$ and $P$ are
orthogonal projections, their $L^2$-norm equals $1$ and we can estimate,
 for $j$ with $|\supp N_j\cap I|>0$,
\begin{equation*}
	|c_j| \lesssim \|f\|_2 \|N_j\|_2 \lesssim \|f\|_\infty |\supp N_j|^{1/2}
	\lesssim \|f\|_\infty |I|^{1/2}.
\end{equation*}
On the other hand, by the above result on the entries $a_{ij}$ of the inverse
matrix to $(\langle N_i, N_j\rangle)$ and \eqref{eq:alphasupp},
\begin{equation*}
	|N_j^*(x)|= |\sum_{i} a_{ij}N_i(x)|\lesssim
	\frac{q^{|j-\ell|}}{{\alpha_j} } \lesssim \frac{q^{|j-\ell|}}{|I|}
\end{equation*}
for $j$ with $|\supp N_j\cap I|>0$, where the index $\ell$ is chosen such that 
$x\in [t_\ell,t_{\ell+1}]$.
Therefore, we get
\[
	\|(\tilde P - P)f\|_\infty \lesssim  |I|^{-1/2} \|f\|_\infty \lesssim
	\varepsilon^{-1/2} \|f\|_\infty
\]
and  we conclude
\[
	\|\tilde{P} f\|_\infty \leq \|Pf\|_\infty + \|(P-\tilde{P})f\|_\infty
	\leq (K + C\varepsilon^{-1/2})\|f\|_\infty,
\]
where $C$ is some constant that depends only on $\mathcal U_k$.

 In order to go from an arbitrary partition $\tilde{\Delta}$ to a partition
$\Delta$ with $\|P_\Delta:L^\infty\to L^\infty\|\leq K_2$, we can apply the above construction
iteratively, until we arrive at a partition $\Delta$ with $|\Delta|\leq
\varepsilon$, where $\varepsilon>0$ is from Theorem \ref{thm:generalwithepsilon}.
Then, it is guaranteed by Theorem \ref{thm:generalwithepsilon}
that $\|P_\Delta:L^\infty\to L^\infty\|\leq K_2$. The number of iteration steps depends only
on $\varepsilon$ and thus, only on $\mathcal U_k$. Therefore, 
\[
	\|\tilde{P} f\|_\infty \lesssim
	\|f\|_\infty,
\] which finishes the proof of the main theorem.
\end{proof}

\subsection*{Acknowledgements } The second author is supported by the FWF-project
Nr.P27723.

 \bibliographystyle{plain}
\bibliography{projection}

\def\cprime{$'$}
\begin{thebibliography}{10}

\bibitem{Ciesielski2000}
Z.~Ciesielski.
\newblock Orthogonal projections onto spline spaces with arbitrary knots.
\newblock In {\em Function spaces ({P}ozna\'n, 1998)}, volume 213 of {\em
  Lecture Notes in Pure and Appl. Math.}, pages 133--140. Dekker, New York,
  2000.

\bibitem{DeVoreLorentz1993}
R.~A. DeVore and G.~G. Lorentz.
\newblock {\em Constructive approximation}, volume 303 of {\em Grundlehren der
  Mathematischen Wissenschaften [Fundamental Principles of Mathematical
  Sciences]}.
\newblock Springer-Verlag, Berlin, 1993.

\bibitem{Golitschek2014}
M.~v. Golitschek.
\newblock On the {$L_\infty$}-norm of the orthogonal projector onto splines.
  {A} short proof of {A}. {S}hadrin's theorem.
\newblock {\em J. Approx. Theory}, 181:30--42, 2014.

\bibitem{Karlin1968}
S.~Karlin.
\newblock {\em Total positivity. {V}ol. {I}}.
\newblock Stanford University Press, Stanford, Calif, 1968.

\bibitem{KeryanPassenbrunner2017}
K.~Keryan and M.~Passenbrunner.
\newblock Unconditionality of periodic orthonormal spline systems in ${L}^p$.
\newblock {\em preprint arXiv:1708.09294, to appear in Studia Math.}, 2017.

\bibitem{MuellerPassenbrunner2017}
P.~F.~X. M\"uller and M.~Passenbrunner.
\newblock Almost everywhere convergence of spline sequences.
\newblock {\em preprint arXiv:1711.01859}, 2017.

\bibitem{Passenbrunner2014}
M.~Passenbrunner.
\newblock Unconditionality of orthogonal spline systems in {$L^p$}.
\newblock {\em Studia Math.}, 222(1):51--86, 2014.

\bibitem{Passenbrunner2017}
M.~Passenbrunner.
\newblock Orthogonal projectors onto spaces of periodic splines.
\newblock {\em Journal of Complexity}, 42:85--93, 2017.

\bibitem{PassenbrunnerShadrin2014}
M.~Passenbrunner and A.~Shadrin.
\newblock On almost everywhere convergence of orthogonal spline projections
  with arbitrary knots.
\newblock {\em J. Approx. Theory}, 180:77--89, 2014.

\bibitem{Schumaker1981}
L.~L. Schumaker.
\newblock {\em Spline functions: basic theory}.
\newblock John Wiley \& Sons Inc., New York, 1981.
\newblock Pure and Applied Mathematics, A Wiley-Interscience Publication.

\bibitem{Shadrin2001}
A.~Shadrin.
\newblock The {$L_\infty$}-norm of the {$L_2$}-spline projector is bounded
  independently of the knot sequence: a proof of de {B}oor's conjecture.
\newblock {\em Acta Math.}, 187(1):59--137, 2001.

\end{thebibliography}

\end{document}